\documentclass{amsart}
\usepackage{amsthm}
\usepackage{latexsym}
\usepackage{dsfont}
\usepackage{bbm}
\usepackage{amssymb}
\usepackage{amsmath}
\usepackage{graphicx}
\usepackage{color}

\numberwithin{equation}{section}

\theoremstyle{plain}
\newtheorem{Theorem}{Theorem}[section]
\newtheorem{Lemma}[Theorem]{Lemma}

\theoremstyle{remark}
\newtheorem{Rem}[Theorem]{Remark}
\theoremstyle{definition}

\DeclareMathOperator{\R}{\mathbb{R}}

\DeclareMathOperator{\Prob}{\mathbb{P}}

\DeclareMathOperator{\E}{\mathbb{E}}

\DeclareMathOperator{\1}{\mathbbm{1}}

\newcommand{\todistr}{\overset{{\rm d}}{\underset{n\to\infty}\longrightarrow}}
\newcommand{\todistrt}{\overset{{\rm d}}{\underset{t\to\infty}\longrightarrow}}
\newcommand{\toweak}{\underset{n\to\infty}{\Rightarrow}}
\newcommand{\toweakt}{\underset{t\to\infty}{\Rightarrow}}

\newcommand{\tp}{\overset{\Prob}{\to}}

\DeclareMathOperator{\N}{\mathbb{N}}

\newcommand{\mm}{\mathcal{Z}}
\newcommand{\mn}{\N}

\newcommand{\mr}{\mathbb{R}}

\begin{document}

\title[Profiles of random recursive trees]{A functional limit theorem for the profile of random recursive trees}
\author{Alexander Iksanov}
\address{Faculty of Computer Science and Cybernetics, Taras Shevchenko National University of Kyiv, 01601 Kyiv, Ukraine}
\email{iksan@univ.kiev.ua}
\author{Zakhar Kabluchko}
\address{Institut f\"{u}r Mathematische Statistik, Westf\"{a}lische Wilhelms-Universit\"{a}t M\"{u}nster,
48149 M\"{u}nster, Germany}
\email{zakhar.kabluchko@uni-muenster.de}
\begin{abstract}
Let $X_n(k)$ be the number of vertices  at level $k$ in a random
recursive tree with $n+1$ vertices. We prove a functional limit
theorem for the vector-valued process $(X_{[n^t]}(1),\ldots,
X_{[n^t]}(k))_{t\geq 0}$, for each $k\in\mathbb N$. We show that
after proper centering and normalization, this process converges
weakly to a vector-valued Gaussian process whose components are
integrated Brownian motions. This result is deduced from a
functional limit theorem for Crump-Mode-Jagers branching processes
generated by increasing random walks with increments that have
finite second moment.
\end{abstract}

\noindent  \maketitle

\section{Introduction and main results}

A (deterministic) {\it recursive tree} with $n$ vertices is a
rooted tree with vertices labeled with $1,2\ldots, n$ that
satisfies the following property: the root is labeled with $1$,
and the labels of the vertices on the unique path from the root to
any other vertex (labeled with $m\in\{2,\ldots, n\}$) form an
increasing sequence. There are $(n-1)!$ different recursive trees
with $n$ vertices, and we denote them $T_{1,n}, T_{2,n},\ldots,
T_{(n-1)!, n}$. A random object $\mathcal{T}_n$ is called {\it
random recursive tree} with $n$ vertices if
$$
\Prob\{\mathcal{T}_n=T_{i,n}\}=\frac{1}{(n-1)!},\quad i=1,2,\ldots, (n-1)!.
$$
A simple way to generate a random recursive tree is as follows. At
time $0$ start with a tree consisting of a single vertex (the
root) labeled  with  $1$. At each time $n$, given a
recursive tree with $n+1$ vertices, choose one vertex uniformly at
random and add to this vertex an offspring labeled by $n$. The
random tree obtained at time $n$ has the same distribution as
$\mathcal{T}_{n+1}$. We refer the reader to Chapter~6 of
\cite{Drmota:2009} for more information.

For $k\in\mn$, let $X_n(k)$ denote the number of vertices at level
$k$ in a random recursive trees on $n+1$ vertices. The level of a
vertex is, by definition, its distance to the root. The root is at
level $0$. The function  $k\mapsto X_n(k)$ is usually referred to
as the \textit{profile} of the tree. In Theorem 3 of
\cite{Fuchs+Hwang+Neininger:2006} it was shown by using analytic
tools that for any fixed $k\in \N$,
\begin{equation}\label{fuchs}
\frac{(k-1)!\sqrt{2k-1}\big(X_n(k)-(\log n)^k/k!\big)}{(\log
n)^{k-1/2}}~\todistr~ {\rm normal}(0,1).
\end{equation}
Furthermore, the uniform in $k=1,2,\ldots, o(\log n)$ rate of
convergence in the uniform metric was obtained. The profiles of
random recursive trees  (along with closely related binary search
trees) have been much studied at the central limit regime levels
$k(n) = \log n + c\sqrt{\log n} + o(\sqrt{\log n})$, $c\in\R$, and
at the large deviation regime levels of the form $k(n)\sim \alpha
n$, $\alpha>0$; see
\cite{chauvin_etal:2001,chauvin_etal:2005,drmota_etal:2008,jabbour:2001,kab_mar_sulz:2017}.
Apart from~\cite{Fuchs+Hwang+Neininger:2006}, we are aware of only
one work studying vertices of random recursive trees at a fixed
level. It is shown in \cite{backhausz_mori:2012} that the
proportion of vertices at level $k\in\N$ having more than $t\log
n$ descendants converges to $(1-t)^k$ a.s. Also, a Poisson limit
theorem is proved in \cite{backhausz_mori:2012} for the number of
vertices at fixed level $k$ that have a fixed number of
descendants.

In this paper we are interested in weak convergence of the random
process $\big(X_{[n^t]}(1),\ldots, X_{[n^t]}(k)\big)_{t\geq 0}$
for each $k\in\mn$, properly normalized and centered, as
$n\to\infty$. The latter vector might be called the {\it low
levels profile}.
\begin{Theorem}\label{main2}
The following functional limit theorem holds for the low levels
profile of a random recursive tree:
\begin{equation}\label{clt5}
\left( \frac{(k-1)!\big(X_{[n^{(\cdot)}]}(k)-((\log n)\cdot)^k/k!
\big)}{(\log n)^{k-1/2}}\right)_{k\in\mn}~\toweak~
\bigg(\int_{[0,\,\cdot]} (\cdot-y)^{k-1}{\rm
d}B(y)\bigg)_{k\in\mn}
\end{equation}
in the product $J_1$-topology on $D^{\N}$, where $(B(u))_{u\geq
0}$ is a standard Brownian motion and $D=D[0,\infty)$ is the
Skorokhod space.
\end{Theorem}
\begin{Rem}
While the stochastic integral $R_1(s):=\int_{[0,\,s]}{\rm
d}B(y)$ on the right-hand side of \eqref{clt5} is interpreted as
$B(s)$, the other stochastic integrals can be defined via
integration by parts which yields
$$R_k(s):=\int_{[0,\,s]}(s-y)^{k-1}{\rm
d}B(y)=(k-1)!\int_0^{s_1}\int_0^{s_2}\ldots\int_0^{s_{k-1}}
B(y){\rm d}y{\rm d}s_{k-1}\ldots{\rm d}s_2$$ for integer $k\geq 2$
and $s\geq 0$, where $s_1=s$. Depending on whether the left- or
right-hand representation is used the latter process is known in
the literature as a Riemann-Liouville process or an integrated
Brownian motion. It can be checked (details can be found in
Section 2 of \cite{Iksanov:2013}) that $R_k(s)$ has the same
distribution as $\sqrt{s^{2k-1}/(2k-1)}B(1)$ for each $s\geq 0$
and $k\in\mn$. In particular, $\E R_k^2(s)=s^{2k-1}/(2k-1)$.
Along similar lines one can also show that
$$\E R_k(s) R_l(u) =\int_0^{u\wedge s}(s-y)^{k-1}(u-y)^{l-1}{\rm
d}y=\begin{cases}
         \sum_{j=0}^{l-1} \binom{l-1}{j}\frac{1}{k+j}
s^{k+j}(u-s)^{l-1-j}, &   \text{if } \ u\geq s\geq 0,   \\
        \sum_{j=0}^{k-1} \binom{k-1}{j}\frac{1}{l+j}
u^{l+j}(s-u)^{k-1-j}, & \text{if} \ 0\leq u<s
\end{cases}$$ for $k,l\in\mn$. Observe that the aforementioned distributional equality  shows that taking in \eqref{clt5} $(\cdot)=1$ and any fixed $k$ we
obtain \eqref{fuchs}. Moreover, taking $(\cdot)=1$ and
$k=1,2,\ldots$, we obtain the following multivariate central limit
theorem for the  low levels profile:
$$
\left(\frac{(k-1)!\big(X_{n}(k)-(\log n)^k/k!
\big)}{(\log n)^{k-1/2}}\right)_{k\in\mn}~\todistr~
(R_k(1))_{k\in\mn}
$$
weakly on $\mathbb R^\mathbb N$ endowed with the product topology, where the limit is a centered Gaussian process with covariance function
$$
\E R_k(1)R_l(1) = \frac {1}{k+l-1}, \quad k,l\in\N.
$$
\end{Rem}

\section{Our approach and an auxiliary tool}

In order to explain our approach that we use to prove Theorem
\ref{main2} we need more notation.

Let $(\xi_k)_{k\in\mn}$ be a sequence of i.i.d.\ positive random
variables with generic copy $\xi$. Denote by $S:=(S_n)_{n\in\mn}$
the ordinary random walk with jumps $\xi_n$ for $n\in\mn$, that
is, $S_n = \xi_1+\ldots+\xi_n$, $n \in \mn$. Further, we define
the renewal process $(N(t))_{t\in\R}$ by
$$N(t):=\sum_{k\geq 1}\1_{\{S_k\leq
t\}},\quad t\in \R.$$ Set $U(t):=\E N(t)$ for $t\in\mr$, so that,
with a slight abuse of terminology, $U$ is the renewal function.
Clearly, $N(t)=0$ a.s.\ and $U(t)=0$ for $t<0$.

Next, we recall the construction of a Crump-Mode-Jagers branching
process in the special case when it is generated by the random
walk $S$. At time $\tau_0=0$ there is one individual, the
ancestor. The ancestor produces offspring (the first generation)
with birth times given by a point process $\mm = \sum_{n\geq 1}
\delta_{S_n}$ on $\R_+:=[0,\infty)$. The first generation produces
the second generation. The shifts of birth times of the second
generation individuals with respect to their mothers' birth times
are distributed according to independent copies of the same point
process $\mm$. The second generation produces the third one, and
so on. All individuals act independently of each other.
Equivalently, one may consider a branching random walk. In this
case, the points of $\mm$ are interpreted as the positions of the
 first generation individuals. Each individual in the
first generation produces individuals from the second generation
whose displacements with respect to the position of their
respective mother are given by an independent copy of $\mm$, and
so on.

For $k\in\mn$, denote by $Y_k(t)$ the number of the $k$th
generation individuals with birth times $\leq t$. Plainly,
$Y_1(t)=N(t)$ for $t\geq 0$. We recall that $0!=1$. For $n\in\mn$,
denote by $\tau_n$ the birth time of the $n$th individual (in the
chronological order of birth times, excluding the ancestor).

Now we are ready to point out the basic observation for the proof
of Theorem \ref{main2}: if $\xi$ has an exponential distribution
of unit mean, then the following  distributional
equality of stochastic processes holds true:
\begin{equation}\label{basic}
(X_{[n^s]}(k))_{s\geq 0, k\in \mn}\overset{{\rm d}}{=}
(Y_k(\tau_{[n^s]}))_{s\geq 0, k\in \mn}.
\end{equation}
In the following, we shall simply identify  these processes.
Formula \eqref{basic} follows from the fact observed by B.\
Pittel, see p.~339 in \cite{Pittel:1994}, that the tree formed by
the individuals in combination with their family relations at time
$\tau_n$ is a version of a random recursive tree with $n+1$
vertices. To give a short explanation, imagine that a random
recursive tree is generated in continuous time as follows. Start
at time $0$ with one  vertex, the root. At any time, any vertex in
the tree generates with intensity $1$ a single offspring, and all
vertices act independently. Then, the birth times of the vertices
at
 the first level form a Poisson point process with
intensity $1$. More generally, if some vertex was born at time
$t$, then the birth times of its offspring minus $t$ form an
independent copy of the Poisson point process. This system can be
identified with the Crump-Mode-Jagers process
generated by an ordinary random walk with jumps having the
exponential distribution of unit mean. 
If $\tau_n$ is the birth time of the $n$th vertex, then the
genealogical tree of the vertices with birth times in the interval
$[0,\tau_n]$ is a random recursive tree. The embedding into a
continuous time process just described was used in~\cite{chauvin_etal:2005, kab_mar_sulz:2017,
Pittel:1994}.

Theorem \ref{main} given next is our main technical tool for
proving Theorem \ref{main2}. We stress that here, the distribution
of $\xi$ is not assumed exponential, so that Theorem \ref{main} is
far more general than what is needed to treat random recursive
trees.
\begin{Theorem}\label{main}
Suppose that $\sigma^2:={\rm Var}\,\xi\in (0, \infty)$. Then
\begin{equation}\label{clt}
\left( \frac{(k-1)!\big(Y_k(t\cdot)-(t\cdot)^k/(k!
\mu^k)\big)}{\sqrt{\sigma^2\mu^{-2k-1}t^{2k-1}}}\right)_{k\in\mn}~\toweakt~
(R_k(\cdot))_{k\in\mn}
\end{equation}
in the product $J_1$-topology on $D^{\N}$, where
$\mu:=\E\xi<\infty$.
\end{Theorem}

For $i\in\mn$, consider the $1$st generation individual born at
time $S_i$ and denote by $Y_j^{(i)}(t)$ for $j\in\mn$ the number
of her successors in the $(j+1)$st generation with
birth times $\leq t+S_i$. By the branching property
$(Y_j^{(1)}(t))_{t\geq 0}$, $(Y_j^{(2)}(t))_{t\geq 0},\ldots$ are
independent copies of $(Y_j(t))_{t\geq 0}$ which are independent
of $S$. With this at hand we are ready to write the basic
representation
$$Y_k(t)=\sum_{i\geq 1}Y^{(i)}_{k-1}(t-S_i),\quad t\geq 0, k\geq 2.$$ Note
that, for $k\geq 2$, $(Y_k(t))_{t\geq 0}$ is a particular instance
of a random process with immigration at the epochs of a renewal
process which is a renewal shot noise process with random and
independent response functions (the term was introduced in
\cite{Iksanov+Marynych+Meiners:2017}; see also \cite{Iksanov:2017}
for a review).

For $t\geq 0$ and $k\in\mn$, set $U_k(t):=\E Y_k(t)$ and observe
that, $U_1(t)=U(t)$ and
$$U_k(t)=\int_{[0,\,t]}U_{k-1}(t-y){\rm d}U(y)=\int_{[0,\,t]}U(t-y){\rm d}U_{k-1}(y).$$ Our strategy of the proof of Theorem \ref{main}
is the following. Using a decomposition
\begin{eqnarray*}
Y_k(t)-\frac{t^k}{k!\mu^k}&=&\sum_{j\geq
1}\big(Y^{(j)}_{k-1}(t-S_j)-U_{k-1}(t-S_j)\1_{\{S_j\leq
t\}}\big)\\&+ &\bigg(\sum_{j\geq 1}U_{k-1}(t-S_j)\1_{\{S_j\leq
t\}}-\mu^{-1}\int_0^t U_{k-1}(y){\rm d}y\bigg)\\&+&
\bigg(\mu^{-1}\int_0^t U_{k-1}(y){\rm
d}y-\frac{t^k}{k!\mu^k}\bigg)=:Y_{k,1}(t)+Y_{k,2}(t)+Y_{k,3}(t)
\end{eqnarray*}
for $k\geq 2$, we shall prove three statements: for all $T>0$,
\begin{equation}\label{probab}
\frac{\sup_{0\leq s\leq T}\,|Y_{k,1}(st)|}{t^{k-1/2}} \tp 0,\quad
t\to\infty;
\end{equation}
\begin{equation}\label{clt4}
\lim_{t\to\infty}t^{-(k-1/2)}\sup_{0\leq s\leq
T}\,|Y_{k,3}(st)|=0,
\end{equation}
and
\begin{equation}\label{clt2}
\left(
\frac{Y_1(t\cdot)-\mu^{-1}(t\cdot)}{\sqrt{\sigma^2\mu^{-3}t}},\frac{(k-1)!Y_{k,2}(t\cdot)}{\sqrt{\sigma^2\mu^{-2k-1}t^{2k-1}}}\right)_{k\geq
2} ~\toweakt~ (R_k(\cdot))_{k\in\mn}
\end{equation}
in the product $J_1$-topology on $D^{\N}$. Plainly,
\eqref{probab}, \eqref{clt4} and \eqref{clt2} entail \eqref{clt}.
Weak convergence of the coordinates in \eqref{clt2} is known: see
Theorem 3.1 on p.~162 in \cite{Gut:2009} for the first coordinate
and Theorem 1.1 in \cite{Iksanov:2013} for the others.

\section{Proof of Theorem \ref{main2}}

Applying Theorem \ref{main}  to exponentially
distributed $\xi$ of unit mean (so that $\mu=\sigma^2=1$) we
obtain
\begin{equation}\label{clt3}
\left( \frac{(k-1)!\big(Y_k((\log n)\cdot)-((\log n) \cdot)^k/k!
\big)}{(\log n)^{k-1/2}}\right)_{k\in\mn}~\toweak~
\big(R_k(\cdot)\big)_{k\in\mn}
\end{equation}
in the product $J_1$-topology on $D^{\mn}$.

It is a classical fact that $\tau_n$ is the sum of $n$ independent
exponentially distributed random variables of means $1,2,\ldots,
n$, whence $\lim_{n\to\infty}(\tau_n/ \log n)=1$ a.s. Arguing as
in the proof of Theorem 3 in \cite{Glynn+Whitt:1988} we conclude
that, for each $T>0$, $\lim_{n\to\infty}\sup_{0\leq s\leq
T}|\tau_{[n^s]}/\log n-\psi(s)|=0$ a.s., where $\psi(s)=s$ for
$s\geq 0$. This in combination with \eqref{clt3} gives
\begin{equation*}
\left(\left( \frac{(k-1)!\big(Y_k(\log n\cdot)-((\log n)
\cdot)^k/k! \big)}{(\log n)^{k-1/2}}\right)_{k\in\mn},
\frac{\tau_{[n^{(\cdot)}]}}{\log n}\right)~\toweak~
\big( \big(R_k(\cdot) \big)_{k\in\mn},
\psi(\cdot)\big)
\end{equation*}
in the product $J_1$-topology on $D^{\mn}\times D$.

It is well-known (see, for instance, Lemma 2.3 on p.~159 in
\cite{Gut:2009}) that, for fixed $j\in\mn$, the composition
mapping $((x_1,\ldots, x_j), \varphi)\mapsto (x_1\circ
\varphi,\ldots, x_j\circ \varphi)$ is continuous at vectors
$(x_1,\ldots, x_j): \mr_+^j\to \mr^j$ with continuous coordinates
and nondecreasing continuous $\varphi: \mr_+\to \mr_+$, where
$\mr_+:=[0,\infty)$. Since $R_k$ is a.s.\ continuous and $\psi$ is
nonnegative, nondecreasing and continuous, 
we can invoke the continuous mapping theorem to infer \eqref{clt5} with $Y_k(\tau_{[n^{(\cdot)}]})$ replacing $X_{[n^{(\cdot)}]}(k)$. In view of \eqref{basic} this completes the proof of Theorem \ref{main2}. 

\section{Proof of Theorem \ref{main}}

It is well known that
\begin{equation}\label{lord} -1\leq U(t)-t/\mu\leq c_0,\quad t\geq 0
\end{equation}
for appropriate constant $c_0>0$ whenever $\E\xi^2<\infty$.
While the left-hand inequality follows from Wald's identity $t\leq
\E S_{N(t)+1}=\mu (U(t)+1)$, the right-hand inequality is Lorden's
inequality (see \cite{Carlson+Nerman:1986} for a short
probabilistic proof in the situation where $\xi$ has a nonlattice
distribution). If the distribution of $\xi$ is nonlattice, one can
take $c_0={\rm Var}\,\xi/\E\xi^2$, whereas if the distribution of
$\xi$ is $\delta$-lattice, \eqref{lord} holds with
$c_0=2\delta/\mu+{\rm Var}\,\xi/\E\xi^2$. We shall need the
following consequence of \eqref{lord}:
\begin{equation}\label{lord2}
|U(t)-t/\mu|\leq c,\quad t\geq 0
\end{equation}
where $c=\max(c_0,1)$.

\begin{Lemma}\label{aux1}
Under the assumption $\E\xi^2<\infty$
\begin{equation}\label{ineq2}
\bigg|U_k(t)-\frac{t^k}{k!\mu^k}\bigg|\leq
\sum_{i=0}^{k-1}\binom{k}{i}\frac{t^i c^{k-i}}{i!\mu^i},\quad
k\in\mn,~t\geq 0.
\end{equation}
\end{Lemma}
\begin{proof}
By using the mathematical induction we first show that
\begin{equation}\label{ineq1}
\bigg|\int_{[0,\,t]}(t-z)^m{\rm
d}U(z)-\frac{t^{m+1}}{(m+1)\mu}\bigg|\leq ct^m,\quad m\in\N_0.
\end{equation} When $m=0$, \eqref{ineq1} is a consequence of
\eqref{lord2}. Assuming that \eqref{ineq1} holds for $m=j-1$ we
obtain $$\bigg|\int_{[0,\,t]}(t-z)^j{\rm
d}U(z)-\frac{t^{j+1}}{(j+1)\mu}\bigg|=\bigg|j\int_0^t\bigg(\int_{[0,\,s]}(s-z)^{j-1}{\rm
d}U(z)-\frac{s^j}{j\mu}\bigg){\rm d}s\bigg|\leq j\int_0^t
cs^{j-1}{\rm d}s=ct^j$$ which completes the proof of
\eqref{ineq1}.

To prove \eqref{ineq2} we once again use the mathematical
induction. When $k=1$, \eqref{ineq2} coincides with \eqref{lord2}.
Assuming that \eqref{ineq2} holds for $k\leq j$ and appealing to
\eqref{ineq1} we infer
\begin{eqnarray*}
&&\bigg|U_{j+1}(t)-\frac{t^{j+1}}{(j+1)!\mu^{j+1}}\bigg|\\&\leq
&\int_{[0,\,t]}\bigg|U_j(t-z)-\frac{(t-z)^j}{j!\mu^j}\bigg|{\rm
d}U(z)+\frac{1}{j!\mu^j}\bigg|\int_{[0,\,t]}(t-z)^j{\rm
d}U(z)-\frac{t^{j+1}}{(j+1)\mu}\bigg|\\&\leq&
\int_{[0,\,t]}\sum_{i=0}^{j-1}\binom{j}{i}\frac{c^{j-i}}{i!\mu^i}(t-z)^i{\rm
d}U(z)+\frac{ct^j}{j!\mu^j}\\&\leq&
\sum_{i=0}^{j-1}\binom{j}{i}\frac{c^{j+1-i}t^i}{i!\mu^i}+\sum_{i=0}^{j-1}\binom{j}{i}\frac{c^{j-i}t^{i+1}}{(i+1)!\mu^{i+1}}+\frac{ct^j}{j!\mu^j}\\&\leq&
c^{j+1}+\sum_{i=1}^{j-1}\bigg(\binom{j}{i}+\binom{j}{i-1}\bigg)\frac{c^{j+1-i}t^i}{i!\mu^i}+\frac{(j+1)ct^j}{j!\mu^j}=
\sum_{i=0}^j\binom{j+1}{i}\frac{c^{j+1-i}t^i}{i!\mu^i}
\end{eqnarray*}
\end{proof}

\begin{Lemma}\label{aux2}
Under the assumption $\E\xi^2<\infty$, for $k\in\mn$,
\begin{equation}\label{ineq3}
D_k(t):={\rm Var}\,Y_k(t)=O(t^{2k-1}),\quad t\to\infty
\end{equation}
and,  for $k\geq 2$,
\begin{equation}\label{ineq4}
\E [(Y_{k,1}(t))^2] = O(t^{2k-2}),\quad t\to\infty.
\end{equation}
\end{Lemma}
\begin{proof}
Using a decomposition
\begin{eqnarray*}
Y_k(t)-U_k(t)&=&\sum_{j\geq
1}\big(Y^{(j)}_{k-1}(t-S_j)-U_{k-1}(t-S_j)\big)\1_{\{S_j\leq
t\}}\\&+& \bigg(\sum_{j\geq 1}U_{k-1}(t-S_j)\1_{\{S_j\leq
t\}}-U_k(t)\bigg)=:Y_{k,1}(t)+Y^\ast_{k,2}(t)
\end{eqnarray*}
we infer
\begin{equation}\label{aux5}
D_k(t)=\E [(Y_{k,1}(t))^2]+\E
[(Y^\ast_{k,2}(t))^2]. 
\end{equation}

We start by proving the asymptotic relation
\begin{eqnarray}\label{aux3}
\E [(Y^\ast_{k,2}(t))^2]&=&{\rm Var}\,\bigg(\sum_{i\geq
1}U_{k-1}(t-S_i)\1_{\{S_i\leq
t\}}\bigg)\notag\\&=&\E\bigg(\sum_{i\geq
1}U_{k-1}(t-S_i)\1_{\{S_i\leq
t\}}\bigg)^2 - U_k^2(t) = O(t^{2k-1}),\quad t\to\infty
\end{eqnarray}
for $k\geq 2$. To this end, we need the following formula
\begin{equation}\label{mom}
\E\bigg(\sum_{i\geq 1}U_{k-1}(t-S_i)\1_{\{S_i\leq
t\}}\bigg)^2 = 2\int_{[0,\,t]} U_{k-1}(t-y)U_k(t-y){\rm
d}U(y)+\int_{[0,\,t]}U_{k-1}^2(t-y){\rm d}U(y).
\end{equation}

\noindent {\sc Proof of \eqref{mom}}. Write
$$\E \bigg(\sum_{i\geq 1}U_{k-1}(t-S_i)\1_{\{S_i\leq
t\}}\bigg)^2=2\E \sum_{1\leq
i<j}U_{k-1}(t-S_i)U_{k-1}(t-S_j)\1_{\{S_j\leq t\}}+\E\sum_{i\geq
1}U_{k-1}^2(t-S_i)\1_{\{S_i\leq t\}}.$$ It is clear that the
second expectation is equal to the second summand on the
right-hand side of \eqref{mom}. Thus, it remains to show that the
first expectation is equal to the first summand on the right-hand
side of \eqref{mom}: 
\begin{eqnarray*}
&&\E \sum_{1\leq i<j}U_{k-1}(t-S_i)U_{k-1}(t-S_j)\1_{\{S_j\leq
t\}}\\&=& \E \sum_{i\geq 1}
U_{k-1}(t-S_i)\big(U_{k-1}(t-S_{i+1})\1_{\{S_{i+1}\leq
t\}}+U_{k-1}(t-S_{i+2})\1_{\{S_{i+2}\leq t\}}+\ldots\big)\\&=&\E
\sum_{i\geq 1} U_{k-1}(t-S_i)\1_{\{S_i\leq
t\}}\E\big(U_{k-1}(t-S_i-\xi_{i+1})\1_{\{\xi_{i+1}\leq
t-S_i\}}\\&+&U_{k-1}(t-S_i-\xi_{i+1}-\xi_{i+2})\1_{\{\xi_{i+1}+\xi_{i+2}\leq
t-S_i\}}+\ldots|S_i\big)\\&=&\E\sum_{i\geq 1}
U_{k-1}(t-S_i)\int_{[0,\,t-S_i]}U_{k-1}(t-S_i-y){\rm
d}U(y)\1_{\{S_i\leq t\}}=\E\sum_{i\geq 1}
U_{k-1}(t-S_i)U_k(t-S_i)\1_{\{S_i\leq
t\}}\\&=&\int_{[0,\,t]}U_{k-1}(t-y)U_k(t-y){\rm d}U(y).
\end{eqnarray*}

Before we proceed let us note that \eqref{ineq1} implies that, for
integer $m\leq 2k-3$,
$$\int_{[0,\,t]}(t-y)^m{\rm d}U(y)=o(t^{2k-1}),\quad t\to\infty,$$
that $$\int_{[0,\,t]}(t-y)^{2k-2} {\rm
d}U(y)=O(t^{2k-1}),\quad t\to\infty$$ and that
$$\int_{[0,\,t]}(t-y)^{2k-1} {\rm
d}U(y)\leq \frac{t^{2k}}{2k \mu}+c t^{2k-1},\quad t\geq 0.$$ Using these
relations in combination with \eqref{ineq2} yields
$$\E
\Big(\sum_{i\geq 1}U_{k-1}(t-S_i)\1_{\{S_i\leq t\}}\Big)^2\leq
\frac{2}{(k-1)!k!\mu^{2k-1}}\int_{[0,\,t]} (t-y)^{2k-1}{\rm
d}U(y)+O(t^{2k-1})\leq \frac{t^{2k}}{(k!)^2\mu^{2k}}+O(t^{2k-1})$$
 as $t\to\infty$. Further,
$$U_k^2(t)=\frac{t^{2k}}{(k!)^2\mu^{2k}}+\frac{2t^k}{k!\mu^k}\bigg(U_k(t)-\frac{t^k}{k!\mu^k}\bigg)+\bigg(U_k(t)-\frac{t^k}{k!\mu^k}\bigg)^2=
\frac{t^{2k}}{(k!)^2\mu^{2k}}+O(t^{2k-1}),\quad t\to\infty$$
having utilized \eqref{ineq2}. The last two asymptotic relations
entail $$\E [(Y^\ast_{k,2}(t))^2] =\E
\Big(\sum_{i\geq 1}U_{k-1}(t-S_i)\1_{\{S_i\leq t\}}\Big)^2-
U_k^2(t) = O(t^{2k-1}),\quad t\to\infty.$$ The proof of
\eqref{aux3} is complete.

To prove \eqref{ineq3} we shall use the mathematical induction. If
$k=1$, \eqref{ineq3} holds true by Lemma \ref{renewal}. Assume
that \eqref{ineq3} holds for $k=m-1\geq 2$. Then given $\delta>0$
there exist $t_0>0$ and $c_m>0$ such that $D_{m-1}(t)\leq
c_mt^{2m-3}$ whenever $t\geq t_0$. Consequently,
\begin{eqnarray}\label{aux4}
\E [(Y_{m,1}(t))^2]&=&\E\sum_{i\geq
1}D_{m-1}(t-S_i)\1_{\{S_i\leq
t\}}=\int_{[0,\,t-t_0]}D_{m-1}(t-x){\rm
d}U(x)\notag\\&+&\int_{(t-t_0,\,t]}D_{m-1}(t-x){\rm d}U(x)\leq
c_m\int_{[0,\,t-t_0]}(t-x)^{2m-3}{\rm d}U(x)\notag\\&+&\sup_{0\leq
y\leq t_0}\,D_{m-1}(y)(U(t)-U(t-t_0))\notag\\&\leq&
c_mt^{2m-3}U(t)+\sup_{0\leq y\leq
t_0}\,D_{m-1}(y)(U(t_0)+1)=O(t^{2m-2}),\quad t\to\infty
\end{eqnarray}
having utilized subadditivity of $U(t)+1$ and the elementary
renewal theorem which states that $U(t)\sim t/\mu$ as
$t\to\infty$. Using \eqref{aux5} and \eqref{aux3} we conclude that
\eqref{ineq3} holds for $k=m$. Relation \eqref{ineq4} is now an
immediate consequence of \eqref{aux4}.
\end{proof}

Now we are ready to prove Theorem \ref{main}.

\begin{proof}[Proof of Theorem \ref{main}]

\noindent {\sc Proof of \eqref{clt4}}. 
In view of \eqref{ineq2} we infer
\begin{eqnarray*}
\mu\sup_{0\leq s\leq T}\,|Y_{k,3}(st)|&\leq& \sup_{0\leq s\leq
T}\,\int_0^{st}\bigg|U_{k-1}(y)-\frac{y^{k-1}}{(k-1)!\mu^{k-1}}\bigg|{\rm
d}y\\&\leq& \sup_{0\leq s\leq T}\,\int_0^{st}
\sum_{i=0}^{k-2}\binom{k-1}{i} \frac{y^ic^{k-1-i}}{i!\mu^i}{\rm
d}y\\&\leq& \sum_{i=0}^{k-2}\binom{k-1}{i}
\frac{(Tt)^{i+1}c^{k-1-i}}{(i+1)!\mu^i}=O(t^{k-1})
\end{eqnarray*}
for all $T>0$. This proves \eqref{clt4}.  

\noindent {\sc Proof of \eqref{probab}}. It suffices to check
that, for integer $k\geq 2$,
\begin{equation}\label{inter}
\lim_{t\to\infty}t^{-(k-1/2)}Y_{k,1}(t)=0\quad\text{a.s.}
\end{equation}
To this end, we pick $\delta\in (1,2)$ and note that for each
$t\ge 0$, there exists $m\in\N_0$ such that $t\in
[m^\delta,(m+1)^\delta)$ and
\begin{eqnarray*}
t^{-(k-1/2)}Y_{k,1}(t) &\le& m^{-\delta(k-1/2)}\sum_{i\ge
1}\big(Y^{(i)}_{k-1}((m+1)^\delta-S_i)-U_{k-1}((m+1)^\delta-S_i)
\1_{\{S_i\le (m+1)^\delta\}}\big)\\
&+& m^{-\delta(k-1/2)}\sum_{i\ge
1}\big(U_{k-1}((m+1)^\delta-S_i)-U_{k-1}(m^\delta-S_i)\big)\1_{\{S_i\le
m^\delta\}}\\&+&m^{-\delta(k-1/2)}\sum_{i\ge
1}U_{k-1}((m+1)^\delta-S_i)\1_{\{m^\delta<S_i\le
(m+1)^\delta\}}\\&\leq&
m^{-\delta(k-1/2)}Y_{k,1}((m+1)^\delta)\\&+&m^{-\delta(k-1/2)}((U((m+1)^\delta-m^\delta)+1)U_{k-2}((m+1)^\delta)N(m^\delta)\\&+&
U_{k-1}((m+1)^\delta-m^\delta)N((m+1)^\delta)))
\end{eqnarray*}
where $U_0(t):=1$ for $t\geq 0$. For the last inequality we have
used monotonicity of the functions $U_i$, $i\in\mn$ and the
following estimate which is essentially based on subadditivity and
monotonicity of $U+1$:
\begin{eqnarray*}
U_i(t+s)-U_i(t)&=&\int_{[0,\,t]}(U(t+s-z)-U(t-z)){\rm
d}U_{i-1}(z)+\int_{(t,\,t+s]} U(t+s-z){\rm d}U_{i-1}(z)\\&\leq&
(U(s)+1)U_{i-1}(t)+U(s)(U_{i-1}(t+s)-U_{i-1}(t))\\&\leq&
(U(s)+1)U_{i-1}(t+s)
\end{eqnarray*}
for $t,s\geq 0$ and $i\geq 2$.

Similarly,
\begin{eqnarray*}
t^{-(k-1/2)}Y_{k,1}(t)&\ge& (m+1)^{-\delta
(k-1/2)}Y_{k,1}(m)\\&-&(m+1)^{-\delta(k-1/2)}((U((m+1)^\delta-m^\delta)+1)U_{k-2}((m+1)^\delta)N(m^\delta)\\&+&U_{k-1}((m+1)^\delta-m^\delta)N((m+1)^\delta)).
\end{eqnarray*}
By the strong law of large numbers for the renewal processes and
Lemma \ref{aux1} $N(m)\sim \mu^{-1}m$ a.s.\ and, for $j\in\mn$,
$U_j(m)\sim \mu^{-j}(j!)^{-1}m^j$ as $m\to\infty$, respectively,
whence, as $m\to\infty$,
$$m^{-\delta(k-1/2)}((U((m+1)^\delta-m^\delta)+1)U_{k-2}((m+1)^\delta)N(m^\delta)~\sim~\frac{\delta}{(k-2)!\mu^k}\frac{1}{m^{1-\delta/2}}\quad\text{a.s.}$$
and
$$m^{-\delta(k-1/2)}U_{k-1}((m+1)^\delta-m^\delta)N((m+1)^\delta)~\sim ~\frac{\delta^{k-1}}{(k-1)!\mu^k}\frac{1}{m^{k-(1+\delta/2)}}\quad
\text{a.s.}$$ Since $\delta<2$ and $k\geq 2$, the right-hand sides
of the last two relations converge to zero a.s. Hence,
\eqref{inter} is a consequence of
\begin{equation}\label{prw_to_zero_as11}
\lim_{\N\ni m\to\infty} m^{-\delta(k-1/2)}Y_{k,1}(m^\delta)\ =\
0\quad\text{a.s.}
\end{equation}

By Markov's inequality in combination with \eqref{ineq4}
$\Prob\{|Y_{k,1}(m^\delta)|>m^{\delta(k-1/2)}\gamma\}=O(m^{-\delta})$
as $m\to\infty$ for all $\gamma>0$ which entails
\eqref{prw_to_zero_as11} by the Borel-Cantelli lemma.

\noindent {\sc Proof of \eqref{clt2}}. We already know that the
distributions of the coordinates in \eqref{clt2} are tight. Thus,
it remains to check weak convergence of finite-dimensional
distributions, that is, for all $n\in\mn$, all $0\leq
s_1<s_2<\ldots<s_n<\infty$ and all integer $j\geq 2$
\begin{equation}\label{fd1}
\bigg(\frac{Y_1^\ast(s_it)}{a_1(t)},\frac{Y_{k,2}(s_it)}{a_k(t)}\bigg)_{2\leq
k\leq j,\,1\leq i\leq n }~ \todistrt~ 
(R_k(s_i))_{1\leq k\leq j,\,1\leq i\leq n},
\end{equation}
where $Y_1^\ast(t):=Y_1(t)-\mu^{-1}t$ and
$a_k(t):=\sqrt{\sigma^2\mu^{-2k-1}t^{2k-1}}/(k-1)!$ for $k\in\mn$
(recall that $0!=1$). If $s_1=0$ we have
$Y_1^\ast(s_1t)=Y_{k,2}(s_1t)=R_i(s_1)=0$ a.s.\ for $k\geq 2$ and
$i\in\N$. Hence, in what follows we assume that $s_1>0$.

By Theorem 3.1 on p.~162 in \cite{Gut:2009}
$$\frac{N(t\cdot)-\mu^{-1}(\cdot)}{\sqrt{\sigma^2 \mu^{-3}t}}~\toweakt~ B$$ 
in the $J_1$-topology on $D$. By Skorokhod's representation
theorem there exist versions $\widehat{N}$ and $\widehat{B}$ such
that
\begin{equation}\label{cs}
\lim_{t\to\infty}\sup_{0\leq y\leq
T}\bigg|\frac{\widehat{N}(ty)-\mu^{-1}ty}{\sqrt{\sigma^2\mu^{-3}t}}-\widehat{B}(y)\bigg|=0\quad\text{a.s.}
\end{equation}
for all $T>0$. This implies that \eqref{fd1} is equivalent to
\begin{equation}\label{fd2}
\bigg(\frac{(k-1)!\mu^{k-1} \widehat{V}_k
(t,s_i)}{t^{k-1}}\bigg)_{1\leq k\leq j,\,1\leq i\leq n}~
\todistrt~ (R_k(s_i))_{1\leq k\leq j,\,1\leq i\leq
n},
\end{equation}
where, for $t,y\geq 0$, $\widehat{V}_1(t,y):=\widehat{B}(y)$ and
$\widehat{V}_k(t,y):=\int_{(0,\,y]}\widehat{B}(x){\rm
d}_x(-U_{k-1}(t(y-x))$, $k\geq 2$. As far as the coordinates
involving $\widehat{V}_1$ are concerned the equivalence is an
immediate consequence of \eqref{cs}. As for the other coordinates,
integration by parts yields, for $s>0$ fixed and $k\geq 2$,
\begin{eqnarray*}
\int_{[0,\,st]}\frac{U_{k-1}(st-x)}{t^{k-1}}{\rm
d}_x\frac{\widehat{N}(x)-\mu^{-1}x}{\sqrt{\sigma^2\mu^{-3}t}}&=&\int_{(0,\,s]}\bigg(\frac{\widehat{N}(tx)-\mu^{-1}tx}{\sqrt{\sigma^2\mu^{-3}t}}
-\widehat{B}(x)\bigg){\rm
d}_x\frac{-U_{k-1}(t(s-x))}{t^{k-1}}\\&+&\int_{(0,\,s]}\widehat{B}(x){\rm
d}_x\frac{-U_{k-1}(t(s-x))}{t^{k-1}}.
\end{eqnarray*}
Denoting by $J(t)$ the first term on the right-hand side, we infer
$$|J(t)|\leq \sup_{0\leq y\leq s}
\bigg|\frac{\widehat{N}(ty)-\mu^{-1}ty}{\sqrt{\sigma^2\mu^{-3}t}}-\widehat{B}(y)\bigg|
(t^{-(k-1)}U_{k-1}(st))$$ which tends to zero a.s.\ as
$t\to\infty$ in view of \eqref{cs} and Lemma \ref{aux1} which
implies that $\lim_{t\to\infty}
t^{-(k-1)}U_{k-1}(st)=s^{k-1}/((k-1)!\mu^{k-1})$.

For $t,y\geq 0$, set $V_1(t,y):=B(y)$ and
$V_k(t,y):=\int_{(0,\,y]}B(x){\rm d}_x(-U_{k-1}(t(y-x))$, $k\geq
2$. We note that \eqref{fd2} is equivalent to
\begin{equation}\label{fd3}
\bigg(\frac{(k-1)!\mu^{k-1} V_k (t,s_i)}{t^{k-1}}\bigg)_{1\leq
k\leq j,\,1\leq i\leq n}~\todistrt 
~(R_k(s_i))_{1\leq k\leq j,\,1\leq i\leq n}
\end{equation}
because the left-hand sides of \eqref{fd2} and \eqref{fd3} have
the same distribution. Both the limit and the converging vectors
in \eqref{fd3} are Gaussian. Hence, it suffices to prove that
\begin{eqnarray}\label{cova}
\lim_{t\to\infty} t^{-(k+l-2)}\E
V_k(t,s)V_l(t,u)&=&\frac{1}{(k-1)!(l-1)!\mu^{k+l-2}}\E
R_k(s)R_l(u)\\&=&\frac{1}{(k-1)!(l-1)!\mu^{k+l-2}}\int_0^{s\wedge
u}(s-y)^{k-1}(u-y)^{l-1}{\rm d}y\notag
\end{eqnarray}
for $k,l\in\mn$ and $s,u>0$. We only consider the cases where
$0<s\leq u$ and $k,l\geq 2$, the case $s>u$ being similar and the
cases where $k$ or/and $l$ is/are equal to $1$ being simpler.

We start by writing
\begin{eqnarray*}
\E V_k(t,s)V_l(t,u)&=&\int_0^s U_{k-1}(t(s-y))U_{l-1}(t(u-y)){\rm
d}y\\&=&\int_0^s
\bigg(U_{k-1}(t(s-y))-\frac{t^{k-1}(s-y)^{k-1}}{(k-1)!\mu^{k-1}}\bigg)U_{l-1}(t(u-y)){\rm
d}y\\&+&\frac{t^{k-1}}{(k-1)!\mu^{k-1}}\int_0^s
(s-y)^{k-1}\bigg(U_{l-1}(t(u-y))-\frac{t^{l-1}(u-y)^{l-1}}{(l-1)!\mu^{l-1}}\bigg){\rm
d}y\\&+& \frac{t^{k+l-2}}{(k-1)!(l-1)!\mu^{k+l-2}}\int_0^s
(s-y)^{k-1}(u-y)^{l-1}{\rm d}y.
\end{eqnarray*}
Denoting by $J_1(t)$ and $J_2(t)$ the first and the second summand
on the right-hand side, respectively, we infer with the help of
Lemma \ref{aux1}:
\begin{eqnarray*}
J_1(t)&\leq& \int_0^s
\sum_{i=0}^{k-2}\binom{k-1}{i}\frac{t^i(s-y)^i}{i!\mu^i}U_{l-1}(t(u-y)){\rm
d}y\\&\leq& U_{l-1}(tu) \sum_{i=0}^{k-2}\binom{k-1}{i}\frac{t^i
s^{i+1}}{(i+1)!\mu^i}=O(t^{k+l-3})
\end{eqnarray*}
as $t\to\infty$ because the sum is $O(t^{k-2})$ and
$U_{l-1}(tu)=O(t^{l-1})$. Arguing similarly we obtain
$J_2(t)=O(t^{k+l-3})$ as $t\to\infty$, and \eqref{cova} follows.
The proof of Theorem \ref{main} is complete.
\end{proof}

\section{Appendix}

Lemma \ref{renewal} is stated in a greater generality than we need
in the present paper because we believe that this result is of
some importance for the renewal theory.

\begin{Lemma}\label{renewal}

\noindent Assume that the distribution of $\xi$ is nondegenerate
and $\E\xi^p<\infty$ for some $p\geq 2$. Then $\E
|N(t)-U(t)|^p\sim \E|Z|^pt^{p/2}$ as $t\to\infty$, where $Z$ is a
normally distributed random variable with mean zero and variance
$\sigma^2\mu^{-3}$,  $\mu=\E \xi$ and $\sigma^2={\rm
Var}\,\xi$.
\end{Lemma}
\begin{proof}
Theorem 8.4 on p.~98 in \cite{Gut:2009} states the result holds
with $\mu^{-1}t$ replacing $U(t)$. Using the inequality (see
p.~282 in \cite{Gut:1974}) $(a+b)^p\leq
a^p+p2^{p-1}(ab^{p-1}+ab^{p-1})+b^p$ for $a, b\geq 0$ together
with $\E|X|\leq (\E|X|^p)^{1/p}$ yields
\begin{eqnarray*}
\E|N(t)-U(t)|^p&\leq&\E|N(t)-\mu^{-1}t|^p+p2^{p-1}(\E|N(t)-\mu^{-1}t|^p)^{1/p}(U(t)-\mu^{-1}t)^{p-1}\\&+&p2^{p-1}\E|N(t)-\mu^{-1}t|^{p-1}
(U(t)-\mu^{-1}t)+(U(t)-\mu^{-1}t)^p.
\end{eqnarray*}
Recalling \eqref{lord} we arrive at
$\lim\sup_{t\to\infty}t^{-p/2}\E|N(t)-\mu^{-1}t|^p\leq \E|Z|^p$.
The converse inequality for the lower limit follows from the
central limit theorem for $N(t)$, formula \eqref{lord} and Fatou's
lemma.
\end{proof}
\begin{Rem}
It is worth stating explicitly that when $p>2$ the assumption
$\E\xi^p<\infty$ in Lemma \ref{renewal} cannot be dispensed with.
According to Remark 1.2 in \cite{Iksanov+Marynych+Meiners:2016},
there exist distributions of $\xi$ such that $\E\xi^2<\infty$ and
$\lim_{t\to\infty}\,t^{-p/2}\E|N(t)-U(t)|^p=\infty$ for every
$p>2$.
\end{Rem}

\vspace{1cm} \noindent   {\bf Acknowledgements}  A
part of this work was done while A.~Iksanov was visiting
M\"{u}nster in late November 2016. He gratefully acknowledges
hospitality and the financial support by DFG SFB 878 ``Geometry,
Groups and Actions''. The authors are grateful to Henning
Sulzbach and Alexander Marynych for useful discussions and  pointers 
to the literature. \quad
\footnotesize 

\end{document}